\documentclass[12pt]{amsart}
\usepackage{amsmath,amsthm,amsfonts,latexsym,amssymb,amscd,color}
\usepackage{chemarr}
\usepackage{mathtools}
\usepackage{manfnt}

\newtheorem{thm}{Theorem}[section]
\newtheorem{cor}[thm]{Corollary}
\newtheorem{lm}[thm]{Lemma}

\newtheorem*{clm*}{Claim}
\theoremstyle{definition}
\newtheorem{df}[thm]{Definition}

\numberwithin{equation}{section}
\newcommand{\sprf}{\noindent{\it Proof.}} 
\newcommand{\sqed}{\hfill\rule{1.3mm}{3mm}\medskip}

\newcommand{\cproof}{\noindent{\it Proof of claim.}\ } 
\newcommand{\cqed}{\hfill\rule{1.3mm}{3mm}}

\newcommand{\m}[1]{{\mathbf{\uppercase{#1}}}}

\DeclareMathOperator{\Con}{Con}

\newcommand{\bd}{\begin{description}}
\newcommand{\ed}{\end{description}}
\begin{document}

\title{Relatively congruence modular quasivarieties of modules}

\author{Keith A. Kearnes}
\address[Keith Kearnes]{Department of Mathematics\\
University of Colorado\\
Boulder, CO 80309-0395\\
USA}
\email{Keith.Kearnes@Colorado.EDU}

\date{May 3, 2015}

\begin{abstract}
  We show that the quasiequational theory of a relatively
  congruence modular quasivariety of left $R$-modules is determined
  by a two-sided ideal in $R$ together with a filter of left ideals.
  The two-sided ideal encodes the identities that hold in the quasivariety,
  while the filter of left ideals encodes the quasiidentities.
  The filter of left ideals defines a generalized notion of torsion.
  
  It follows from our result that if $R$ is left Artinian,
  then any relatively
  congruence modular quasivariety of left $R$-modules is axiomatizable
  by a set of identities together with at most one proper
  quasiidentity, and if $R$ is a commutative Artinian ring then
  any relatively
  congruence modular quasivariety of left $R$-modules is a variety.
\end{abstract}

\maketitle

\begin{center}
{\it Dedicated to Don Pigozzi}
\end{center}

\section{Introduction}\label{intro}
This paper is inspired by
Problem 9.13 of Don Pigozzi's paper,
{\it Finite basis theorems for relatively congruence-distributive
  quasivarieties},
\cite{pigozzi88}.
The problem asks:
\begin{quote}
  {\it
    Is it true that every
    finitely generated and
    relatively congruence modular quasivariety
    is finitely based?}
\end{quote}
This question is still open.
Pigozzi's paper shows the answer to be 
affirmative if ``modular'' is strengthened to ``distributive''.
Analogous problems for varieties were shown
to have positive solutions in
\cite{baker77,mckenzie87,willard00,kearnes15},
and a related problem for quasivarieties
was shown to have an affirmative solution in
\cite{maroti04}.
The best partial answer to Problem~9.13 that now exists
is Theorem 8 of \cite{dziobiak09}, which says  that
if a quasivariety
$\mathcal K$ {\it and the variety it generates} are 
    finitely generated and
    relatively congruence modular,
    then $\mathcal K$
    is finitely based.

A relatively congruence distributive quasivariety
is nothing other than a relatively congruence modular
quasivariety in which no member has a nontrivial
abelian congruence, so Pigozzi's paper solves the
part of Problem 9.13 that does not involve abelian
congruences. Tools for dealing with abelian
congruences in relatively congruence modular quasivarieties
were developed in \cite{kearnes92,kearnes98},
but they have not yet yielded a full solution to the problem.
What these tools show is that abelian congruences in
such quasivarieties are quasiaffine,
which means the blocks of an abelian congruence support
a module-like structure. In this paper we study
the purest relatively congruence modular
quasivarieties which are not distributive,
namely quasivarieties of modules.
Our main result is a description of all relatively
congruence modular quasivarieties of modules.

\section{the classification theorem}

For a unital ring $R$ let $R$-Mod be the variety
of left $R$-modules.
If $\mathcal K$ is a subquasivariety of $R$-Mod
and $M\in \mathcal K$, then a {\it $\mathcal K$-submodule}
(or {\it relative submodule}) of $M$ is an $R$-submodule $S\leq M$
such that $M/S\in \mathcal K$. $\mathcal K$ is
{\it relatively congruence
  modular (RCM)} if every $M\in \mathcal K$
has a modular lattice of $\mathcal K$-submodules.

For the simplest example of these definitions
take $R=\mathbb Z$, so that $R$-Mod is the variety of abelian
groups. Let $\mathcal K$ be the subquasivariety
of $\mathbb Z$-Mod consisting of torsion-free
abelian groups. The only relative submodules on, say,
$\mathbb Z\in \mathcal K$ are $(0)$ and $\mathbb Z$.
That is, $\mathbb Z$ is {\it relatively simple}
(although $\mathbb Z$ is far from being a simple module).
It can be shown that this quasivariety,
the quasivariety of torsion-free
abelian groups, is a minimal quasivariety which happens to be RCM.

Let's generalize the example above in an artificial way.
Let $R = \mathbb Z[t]$ and let $\mathcal K$
be the subquasivariety of $R$-Mod consisting
of all torsion-free abelian groups considered
as $R$-modules by defining $t$ to act as zero
on any module in $\mathcal K$
Then $\mathcal K$ is axiomatized by the identity
$tx=0$ together with a family of quasiidentities
of the form
\[nx=0\to x=0.\]
This is essentially the same as
the preceding example, so in particular it is RCM.

The point of this paper is to show
that every RCM quasivariety of modules looks like the one
from the previous paragraph.
For any RCM quasivariety 
$\mathcal K$ of $R$-modules
there is a set $\Sigma$ of 1-variable identities
along with a specific torsion notion which realizes
$\mathcal K$ as the subquasivariety of $R$-Mod
consisting of the torsion-free
$R$-modules that satisfy $\Sigma$.
$\Sigma$ corresponds to a two-sided ideal in $R$
while the torsion notion corresponds to a filter in the lattice
of left ideals of $R$.

Let's begin by identifying the role played by $\Sigma$.

  \begin{lm}\label{lm1}
    Let $\mathcal V$ be a subvariety of $R$-Mod and let
    $I=\{r\in R\;|\;\mathcal V\models rx=0\}$ be its annihilator.
    Then
    \begin{enumerate}
    \item[(1)] $I$ is a two-sided ideal in $R$,
      \item[(2)] $\Sigma:=\{rx=0\;|\;r\in I\}$
        axiomatizes $\mathcal V$ relative to $R$-Mod, and
      \item[(3)] $\mathcal V$ is definitionally equivalent to $R/I$-Mod.
        \end{enumerate}
  \end{lm}
  We imagine applying this in the situation where
  $\mathcal K$ is a subquasivariety of $R$-Mod
  and $\mathcal V$ is the variety generated by $\mathcal K$.

  We do not prove Lemma~\ref{lm1}, but do point out that a key idea
  in the proof is that a single module identity
  $r_1x_1+\cdots + r_kx_k=0$ has the same strength as
  the set $\{r_1x_1=0, \ldots, r_kx_k=0\}$
  of 1-variable module identities.

Lemma~\ref{lm1} allows us to pass from $R$ to $R/I$
and henceforth consider only the situation where
$\mathcal K$ generates $R$-Mod. We shall make this assumption
as we work out the main result of the paper.

Next we describe the torsion concept that plays a role
in this paper.
\begin{df}\label{torsion}
  Let $\mathcal L$ be the poset of finitely generated
left ideals of $R$ ordered by inclusion.
A {\it torsion notion} for $R$
is a subset $\mathcal F\subseteq \mathcal L$ satisfying the
following conditions:
\begin{enumerate}
\item[(1)] $\mathcal F$ is a nonempty order filter in $\mathcal L$.
  ($A\in \mathcal F$, $B\in \mathcal L$, and $A\subseteq B$
  implies $B\in \mathcal F$.)
  \item[(2)] $\mathcal F$ is downward directed.
    ($A, B\in \mathcal F$ implies there is a $C\in \mathcal F$
    such that $C\subseteq A$ and $C\subseteq B$.)
  \item[(3)] If $X, Y\subseteq R$
    are finite subsets such that the left ideals
    $(X)$ and $(Y)$ belong to $\mathcal F$, 
    then the left ideal $(XY)$ belongs to $\mathcal F$.
  \item[(4)] For all $A\in \mathcal F$ and $r\in R$
    there is $B\in\mathcal F$ such that $Br\subseteq A$.
  \item[(5)] (Regularity of elements of $\mathcal F$)
    If $A\in\mathcal F$, $r\in R$, and $Ar=0$,
    then $r=0$.
\end{enumerate}
\end{df}
Given a torsion notion $\mathcal F$ we say that
an element $m$ of an $R$-module $M$ is an
{\it $\mathcal F$-torsion element}
if $Am=0$ for some $A\in\mathcal F$. If $M$ has no nonzero
$\mathcal F$-torsion elements, then it is {\it $\mathcal F$-torsion-free}.
This may also be expressed by saying that $M\models Ax=0\to x=0$
for each $A\in\mathcal F$. It is easy to see that a
statement of the form $Ax=0\to x=0$
for a finitely generated left ideal $A$ is equivalent
to a quasiidentity. Namely
if $A=(a_1,\ldots,a_m)$,
then $Ax=0\to x=0$ is satisfied if and only if 
\[
(a_1x=0)\wedge \cdots \wedge (a_mx=0) \to (x=0) \tag{$q_A$}
\]
is satisfied, 
so the class of $\mathcal F$-torsion-free $R$-modules is
a quasivariety. 

Any torsion notion $\mathcal F$ contains $R$, and the set
$\{R\}$ is always a torsion notion.
Every $R$-module is torsion-free with respect to this
trivial torsion notion.

Items (1)--(4) of Definition~\ref{torsion}
simplify quite a bit when $R$ is commutative.
Namely, (4) automatically holds when $R$ is commutative, since
we can choose $B=A$. Item (3) now asserts that $\mathcal F$
is closed under multiplication. When this holds, (2) will also
hold, since for commutative rings the product of two ideals
is contained in each of them. Thus (1)--(4) merely say that
$\mathcal F$ is a multiplicatively closed order filter in the poset of
finitely generated ideals. (For any ring $R$,
item (5) of the definition asserts
that the free $R$-module, $R$, is $\mathcal F$-torsion-free.)

Here is the statement of the main theorem of the paper.

\begin{thm}\label{main}
  Let $\mathcal K$ be a quasivariety
  of $R$-modules such that the variety generated by $\mathcal K$
  is all of $R$-Mod. Then $\mathcal K$ is RCM iff
  there is a torsion notion $\mathcal F$ such that
  $\mathcal K$ is the quasivariety of $\mathcal F$-torsion-free
  $R$-modules.
  \end{thm}

We prove Theorem~\ref{main} in the next two sections,
but here we derive a corollary.

\begin{cor}\label{cor1}
  Let $R$ be a left Artinian ring.
  If $\mathcal K$ is an RCM quasivariety of $R$-modules,
  then $\mathcal K$ may be axiomatized relative to $R$-Mod
  by a finite set of identities together with at most one proper
  quasiidentity. In particular, if $R$ is a finitely
  presentable ring, then $\mathcal K$ is finitely axiomatizable. 
\end{cor}

\begin{proof}
  By the Hopkins-Levitski Theorem, a left Artinian ring is
  left Noetherian,
  so every left ideal of $R$ is finitely generated. In particular,
  we can effect the passage from $R$-modules to $R/I$-modules
  (as indicated in Lemma~\ref{lm1}) by imposing finitely
  many identities on the variety $R$-Mod. Thus we may
  assume henceforth that $\mathcal K$ generates $R$-Mod as a variety
  and our goal is now to prove that $\mathcal K$ can be axiomatized
  relative to $R$-Mod by at most one quasiidentity.

  Let $\mathcal F$ be the torsion notion guaranteed
  by Theorem~\ref{main}. Since $R$ is Artinian, $\mathcal F$
  is generated as an order filter by its minimal elements.
  By item (2) of Definition~\ref{torsion},
  $\mathcal F$ is a principal order filter in $\mathcal L$, say
  $\mathcal F$ is the order filter generated
  by $A\in\mathcal L$.
  Now the notion of `$\mathcal F$-torsion-free' is expressible
  by $Ax=0\to x=0$, or equivalently by the
  single quasiidentity $q_A$. (What has been left unsaid so far
    is that if $A\subseteq B$, then $Ax=0\to x=0$ is stronger
  than $Bx=0\to x=0$.)

  The last assertion of the corollary follows from the fact
  that $R$-Mod is finitely axiomatizable when $R$ is finitely presentable.
  \end{proof}

The one proper quasiidentity mentioned in Corollary~\ref{cor1}
can be eliminated when $R$ is commutative.

\begin{cor}\label{cor2}
  Let $R$ be a commutative Artinian ring.
  Any RCM quasivariety of $R$-modules is a variety.
\end{cor}

\begin{proof}
  Let $\mathcal K$ be an RCM quasivariety of $R$-modules.
  We shall argue the proof for arbitrary $R$ until
  it is necessary to appeal to commutativity.
  
Using Lemma~\ref{lm1} we may reduce to the case where
the variety generated
  by $\mathcal K$ is all of $R$-Mod. 
  In the proof of Corollary~\ref{cor1} we showed that
the torsion notion $\mathcal F$ associated 
to $\mathcal K$ is a principal filter in the poset of
finitely generated left
ideals of $R$. Let $A$ be the generator of this
principal filter.

Item (4) of the definition of
`torsion notion' implies that $A$ is a two-sided ideal.
For if $r\in R$, then there must be a $B\in \mathcal F$
such that $Br\subseteq A$. Since $A\subseteq B$ this yields
$Ar\subseteq Br\subseteq A$.

Another special property
that $A$ must satisfy is that $A^2=A$.
To see this, choose a finite set $X$ that generates $A$
as a left ideal. Then item (3) implies that $(X^2)\in \mathcal F$.
But clearly $(X^2)\subseteq A^2$, 
so $A\subseteq (X^2)\subseteq A^2\subseteq A$.

Now we invoke the commutativity hypothesis.
A finitely generated idempotent ideal is generated
by an idempotent element, so $A = (e)$ for some element
$e$ satisfying $e^2=e$. For $r = 1-e$ we have $Ar=0$,
so item (5) of the definition of `torsion notion'
yields that $1-e=0$, i.e. $e=1$, or equivalently $A=R$.
This forces $\mathcal F = \{R\}$. As noted
after Definition~\ref{torsion}, this implies
that every $R$-module is $\mathcal F$-torsion-free,
so $\mathcal K$ = $R$-Mod.
  \end{proof}

Corollaries~\ref{cor1} and \ref{cor2}
are not true if you weaken `Artinian' to `Noetherian',
since the quasivariety of torsion-free abelian groups is not finitely
axiomatizable and is not a variety.
Also, Corollary~\ref{cor2} is not true without the commutativity hypothesis.
To see this, let $R$ be the
ring of upper triangular $2\times 2$
matrices over some field. If the matrix units in $R$
are $e_{11}, e_{12}, e_{22}$, then the quasivariety of $R$-modules
axiomatized relative to $R$-Mod by
$
(e_{11}x=0)\wedge (e_{12}x=0) \to (x=0)
$
is RCM and is not a variety.

%

\section{RCM $\Longrightarrow$ torsion notion}
In this section we prove that if $\mathcal K$
is an RCM quasivariety of $R$-modules and the variety
generated by $\mathcal K$ is all of $R$-Mod, then
there is a torsion notion $\mathcal F$ such that
$\mathcal K$ is the quasivariety of $\mathcal F$-torsion-free
$R$-modules. This is one direction of the proof
of Theorem~\ref{main}.

To prove what is needed we make use of the fact,
proved in \cite{kearnes92}, that
an RCM quasivariety has an `almost equational
axiomatization', and that the $\mathcal K$-extension
of a submodule can be computed easily with the aid
of that axiomatization. Here the {\it $\mathcal K$-extension} of
  a submodule $S\leq M$ is the least
  $\mathcal K$-submodule $\overline{S}\leq M$ that contains $S$.

We recall the necessary concept from \cite{kearnes92}.
A {\it $\Delta$-axiom} is a first-order sentence,
involving pairs of terms
$(p_j(x,y,\bar{u}, \bar{v}, \bar{z}),q_j(x,y,\bar{u}, \bar{v}, \bar{z}))$,
$j<n$, expressing that
\begin{enumerate}
\item the identities 
\[
p_j(x,x,\bar{u}, \bar{u}, \bar{z})=
q_j(x,x,\bar{u}, \bar{u}, \bar{z})
\]
hold for $j<n$, and
\item
  the quasiidentity
\[
\bigwedge (p_j(x,y,\bar{u}, \bar{u}, \bar{z})=
q_j(x,y,\bar{u}, \bar{u}, \bar{z})))\to (x=y)
\]
holds.
\end{enumerate}
We label this $\Delta$-axiom $\Delta(p,q)$.

The two theorems from 
\cite{kearnes92} that we will use are:

\begin{thm} \label{thm5.1}
  (Theorem 5.1 of \cite{kearnes92})
  Let $\mathcal K$ be an RCM quasivariety. $\mathcal K$
is axiomatized
by a set of $\Delta$-axioms combined with a set of identities.
\end{thm}

For the following theorem, the $\mathcal K$-extension of a congruence
$\theta$ is the least $\mathcal K$-congruence containing $\theta$.

\begin{thm} \label{thm5.2}
    (Theorem 5.2 of \cite{kearnes92})
  Let $\mathcal K$ be an RCM quasivariety.
Let $\m a\in {\mathcal K}$, $\theta\in\Con(\m a)$, and $u, v\in A$.
Then $(u,v)$ belongs to the $\mathcal K$-extension of $\theta$
iff there is some $\Delta$-axiom $\Delta(p,q)$ valid in $\mathcal K$,
some pairs $(a_i,b_i)\in\theta$, and some elements $\bar{c}$ such that
$p_i(u,v,\bar{a}, \bar{b}, \bar{c})=q_i(u,v,\bar{a}, \bar{b}, \bar{c})$
for all $i$.
\end{thm}

When dealing with quasivarieties of modules it is possible
to code a $\Delta$-axiom $\Delta(p,q)$
as a left ideal in such a way that
the following are true.

\begin{thm}\label{delta4mod}
  Let $\Delta(p,q)$ be a $\Delta$-axiom
  and let $A$ be its encoding as a left ideal.
\begin{enumerate}
\item An $R$-module
$M$ satisfies $\Delta(p,q)$ iff it satisfies
  $Ax=0\to x=0$.
\item
  If $\mathcal K$ is an RCM quasivariety
of $R$-modules, $M\in\mathcal K$, $S\leq M$ is a submodule,
and $\overline{S}$ is its $\mathcal K$-extension,
then $m\in M$ can be shown to lie in $\overline{S}$
using $\Delta(p,q)$ (in the way described in Theorem~\ref{thm5.2}) 
iff $Am\subseteq S$.
\end{enumerate}
\end{thm}

In order to prove the theorem we must first describe
how to encode a $\Delta$-axiom as a left ideal.

The first step of the construction
uses the fact that equations of the form $p=q$
can be rewritten as equations of the form $(p-q)=0$.
So take a $\Delta$-axiom for $R$-modules, $\Delta(p,q)$,
and rewrite its pairs as differences
\[
\begin{array}{rl}
  D_j(x,y,\bar{u},\bar{v},\bar{z}) &:=
p_j(x,y,\bar{u},\bar{v},\bar{z}) -
q_j(x,y,\bar{u},\bar{v},\bar{z})\\
&=
a_jx+b_jy+\sum_i c_{ij}u_{i}+\sum_i d_{ij}v_{i}+\sum_i e_{ij}z_{i},
\end{array}
\]
where $a_j, b_j, c_{ij}, d_{ij}, e_{ij}\in R$.
Item (1) from the definition of a $\Delta$-axiom now reads
\[
\tag*{$(1)^{\prime}$}
D_j(x,x,\bar{u}, \bar{u}, \bar{z})=0=(a_j+b_j)x+
\sum_i (c_{ij}+d_{ij})u_i + \sum_i e_{ij}z_i.
\]
We will be working in the situation where
$\mathcal K$ is a quasivariety of modules and
the variety it generates is all of $R$-Mod.
For $(1)^{\prime}$ to hold in such a quasivariety the coefficients
in the righthand expression must all be zero, i.e.
$a_j+b_j = c_{ij}+d_{ij} = e_{ij} = 0$. Thus
\[
D_j(x,y,\bar{u}, \bar{v}, \bar{z})=a_j(x-y) + 
\sum_i c_{ij}(u_{i}-v_{i})
\]
(no dependence on the last block of variables).
Introducing new variables $X, U_i$
to represent $x-y, u_i-v_i$ we shall find that the
module term operation
\begin{equation}\label{Eform}
E_j(X,\bar{U}) = a_jX + \sum_i c_{ij}U_i
\end{equation}
can be used to replace the pair $(p_j,q_j)$ in the definition
of `$\Delta$-axiom'. That is, $\Delta(p,q)$ can be rewritten in a reduced
form in an obvious way using the terms $E_j(X,\bar{U})$.

The
left ideal associated to $\Delta(p,q)$
is defined to be $A=(a_0,\ldots,a_{n-1})$,
the left ideal generated by the coefficients of $X$
in the module terms $E_j(X,\bar{U})$, $j<n$.

\begin{proof}[Proof of Theorem~\ref{delta4mod}]
  For part (1) of Theorem~\ref{delta4mod}
  consider a $\Delta$-axiom $\Delta(p,q)$ and write
  it using the terms from (\ref{Eform}).
  Condition $(1)^{\prime}$ of the definition of a $\Delta$-axiom now reads
\[
\tag*{$(1)^{\prime\prime}$}
E_j(0,\bar{0}) = 0,
\]
which must hold simply because $E_j$ is a module term.
Condition (2) of the definition of a $\Delta$-axiom reads
\[
\bigwedge_j (E_j(X,\bar{0}) = 0)\to (X=0).
\]
This is equivalent to $Ax=0\to x=0$ for $x=X$.
This establishes Theorem~\ref{delta4mod}~(1).

Now we turn to Theorem~\ref{delta4mod}~(2).
Suppose that $M\in \mathcal K$, $S\leq M$ and $m\in \overline{S}$.
Choose $\Delta(p,q)$ witnessing that $m\in \overline{S}$.
With $\Delta(p,q)$ written in terms of the $E_j$'s we have
that there exist a tuple $\bar{s}$ with entries in $S$ such that
$E_j(m,\bar{s}) = a_jm+\sum_i c_{ij}s_{ij}$
belongs to $S$ for all $j$. This means that
$a_jm \in -\sum_i c_{ij}s_{ij}+S = S$
for all $j$, or $Am\subseteq S$ for $A$ equal to the
associated left ideal.
Conversely, assume that $Am\subseteq S$. Then for
$\bar{s} = \bar{0}$ we have that $E_j(m,\bar{0})\in S$ for all $j$,
so the form of $\Delta(p,q)$ that uses the terms $E_j(X,\bar{U})$
shows that $m\in\bar{S}$.
This establishes Theorem~\ref{delta4mod}~(2).
\end{proof}

Now we state and prove the main theorem of this section.

\begin{thm}
  Let $\mathcal K$ be an RCM quasivariety such that the variety
  generated by $\mathcal K$ is all of $R$-Mod. If
  $\mathcal F$ is the set of left ideals of $R$
  that code $\Delta$-axioms true in $\mathcal K$,
  then $\mathcal F$ is a torsion notion for $R$-modules
  and 
  $\mathcal K$ is the quasivariety of $\mathcal F$-torsion-free
  $R$-modules.
  \end{thm}

\begin{proof}
  According to Theorem~\ref{thm5.1}, $\mathcal K$ is axiomatized
  relative to $R$-Mod by a set of $\Delta$-axioms together with
  a set of identities. Since the variety generated by $\mathcal K$
  is all of $R$-Mod, we will not use any identities other than
  those that hold in $R$-Mod. According to Theorem~\ref{delta4mod}, the
  $\Delta$-axioms true in $\mathcal K$ are equivalent
  to a family of statements of the form
  $Ax=0\to x=0$ where $A$ is a finitely generated left ideal.
  Let $\mathcal F$ be the set of
  finitely generated left ideals of $R$ such that
  $Ax=0\to x=0$ holds in $\mathcal K$. Since the subset
  of these left ideals that arise from $\Delta$-axioms
  already serves to axiomatize $\mathcal K$ relative
  to $R$-Mod, the full set also serves to axiomatize
  $\mathcal K$ relative
  to $R$-Mod. It follows from this that, if we show that
  $\mathcal F$ is a torsion notion, then $\mathcal K$
  must be the quasivariety of $\mathcal F$-torsion-free $R$-modules.

  Item (1) from definition of a `torsion notion' is the claim
  that $\mathcal F$ is an order filter in the poset of finitely
  generated left ideals of $R$. That is, if $A$ and $B$ are finitely
  generated left ideals of $R$, $A\subseteq B$, and $Ax=0\to x=0$ holds
  in $\mathcal K$,  then $Bx=0\to x=0$
  also holds in $\mathcal K$. This is true because $A\subseteq B$
  implies that  $Bx=0\to Ax=0$.

  Item (5) is the next easiest to verify. Since the variety
  generated by $\mathcal K$ is $R$-Mod, both $\mathcal K$
  and $R$-Mod have the same free modules.
  Hence the 1-generated free module $R$ belongs to $\mathcal K$.
  Hence $R$ satisfies $Ax=0\to x=0$ for each $A\in \mathcal F$,
  which is exactly what (5) claims.

  Item (2) asserts $\mathcal F$ is down directed. Choose $A, B\in\mathcal F$.
  We shall apply Theorem~\ref{delta4mod} to the situation
  $M:=R\oplus R\;(\in\mathcal K)$ and $S:=A\oplus B\leq M$.
  Note that the pair $(1,0)\in M$ belongs to the $\mathcal K$-extension
  of $S$, since $A\in \mathcal F$ and $A(1,0)\subseteq A\oplus B$.
  Similarly $(0,1)$ belongs to the $\mathcal K$-extension of $S$,
  since $B\in\mathcal F$ and $B(0,1)\subseteq A\oplus B$.
  Since $\overline{S}$ is a submodule, the element
  $(1,0)+(0,1)=(1,1)$ must
  belong to the $\mathcal K$-extension of $A\oplus B$. Hence
  there must exist $C\in \mathcal F$ such that $C(1,1)\subseteq A\oplus B$.
  Necessarily $C\subseteq A\cap B$.

  Item (4) asserts that for all $A\in\mathcal F$ and $r\in R$ there is a
  $B\in\mathcal F$ such that $Br\subseteq A$. To prove this we again
  apply the second part of Theorem~\ref{delta4mod}. Let $M=R\in\mathcal K$
  and let $S=A$. The element $1\in R(=M)$ belongs to the $\mathcal K$-extension
  of $A(=S)$, since $A\cdot 1\subseteq S$. The $\mathcal K$-extension of
  $S$ is a submodule, so for any $r\in R$ we have that $r\cdot 1 = r\in M$
  also belongs to the $\mathcal K$ extension of $S=A$.
  Theorem~\ref{delta4mod} guarantees the existence of $B\in \mathcal F$
  such that $B\cdot r\subseteq A$, which is what item (4) requires.

  Item (3) asserts that if $X, Y\subseteq R$
    are finite subsets such that the left ideals
    $(X)$ and $(Y)$ belong to $\mathcal F$, 
    then the left ideal $(XY)$ belongs to $\mathcal F$.
    To prove this, assume that $X = \{a_0,\ldots,a_{m-1}\}$
    and $Y = \{b_0,\ldots,b_{n-1}\}$. The fact that
    $(X), (Y)\in\mathcal F$ implies that $\mathcal K$ satisfies
    the quasiidentities
    \[
    \bigwedge_i   (a_ix=0)\to (x=0) \quad\textrm{and} \quad
    \bigwedge_j   (b_jx=0)\to (x=0).
    \]
    But this means that $\mathcal K$ satisfies
    \begin{equation}\label{xy}
    \bigwedge_j  \left(\bigwedge_i (a_i(b_jx)=0)\right)\to (x=0).
    \end{equation}
    For, if $\bigwedge_i (a_i(b_jx)=0)$ holds for a fixed $j$,
    then the quasiidentity associated to $X$ guarantees that $b_jx=0$.
    But if this holds for all $j$, then the quasiidentity
    associated to $Y$ guarantees that $x=0$.
    Now (\ref{xy}) is just the quasiidentity associated to
    $XY=\{a_ib_j\;|\;i<m, j<n\}$. Since we have shown that it
    holds in $\mathcal K$ we conclude that $(XY)\in\mathcal F$.
  \end{proof}

\section{torsion notion $\Longrightarrow$ RCM}
In this section we prove that if
$\mathcal F$ is a torsion notion for $R$-modules,
then the quasivariety of $\mathcal F$-torsion-free
$R$-modules is RCM and the variety it generates is
all of $R$-Mod.
This is other direction of the proof
of Theorem~\ref{main}.

\begin{lm}\label{extension}
  Assume that
  $\mathcal F$ is a torsion notion for $R$-modules,
  and that $\mathcal K$ is the quasivariety of
  $\mathcal F$-torsion-free $R$-modules.
  If $M\in \mathcal K$ and $S\leq M$ is a submodule
  of $M$, then the $\mathcal K$-extension of $S$ is the set
  \[
\overline{S}:=\{m \in M\;|\; \exists A\in \mathcal F(Am\subseteq S)\}.
  \]
\end{lm}

(In this lemma we are not assuming that $\mathcal K$ is RCM,
so we cannot refer to Theorem~\ref{delta4mod}.)

\begin{proof}
  The set $\overline{S}$ defined in the statement
  contains $S$ because $S$ is a submodule. (One can take
  $A=R\in\mathcal F$ to prove any $m\in S$ belongs to $\overline{S}$.)

  Let's prove that $\overline{S}$ 
  is closed under addition.
  If $x, y\in \overline{S}$, then there exist $A, B\in \mathcal F$
  such that $Ax, By\subseteq S$. By the down directedness
  of $\mathcal F$ there is a $C\subseteq A\cap B$ such that
  $C\in \mathcal F$. For this $C$ we have
  \[C(x+y)\subseteq Cx+Cy\subseteq Ax+By\subseteq S,\]
  yielding $x+y\in \overline{S}$.

  Now we argue that $\overline{S}$ is closed under scalar
  multiplication.
  Assume that $x\in\overline{S}$ and $r\in R$.
  Since $x\in\overline{S}$ 
  there is some $A\in \mathcal F$ such that $Ax\subseteq S$.
  By item (4) of Definition~\ref{torsion} there
  exists $B\in\mathcal F$ such that $Br\subseteq A$. Thus
  \[B(rx)\subseteq Ax\subseteq S,\]
  yielding $rx\in \overline{S}$.

  Next we argue that $\overline{S}$ is a $\mathcal K$-submodule
  of $M$. For this we must show that $M/\overline{S}\in \mathcal K$,
  or that $M/\overline{S}$ is $\mathcal F$-torsion-free.
  This can be established by showing that if $A\in \mathcal F$, 
  $m\in M$, and $Am\subseteq\overline{S}$, then $m\in\overline{S}$.
  Suppose that $A = (a_0,\ldots,a_{m-1})$ as a left ideal.
  The statement $Am\subseteq\overline{S}$ now means
  $\{a_0m, \ldots, a_{m-1}m\}\subseteq \overline{S}$. For each $k$
  there must exist $A_k\in\mathcal F$ such that $A_k(a_km)\subseteq S$.
  By the down directedness of $\mathcal F$ there is a $B\subseteq \cap A_k$,
  and this $B$ has the property that $Ba_km\subseteq S$ for all $k$.
  Suppose that $B=(b_0,\ldots,b_{n-1})$ as a left ideal. By item (3)
  of Definition~\ref{torsion} the left ideal $C$ generated
  by the set $\{b_ja_i\;|\;i<m, j<n\}$ belongs to
  $\mathcal F$. $Cm$ is the submodule of $M$
  generated by all elements $b_ja_im$, all of which belong to $S$.
  Thus $Cm\subseteq S$. This forces $m\in\overline{S}$,
  concluding the proof that $M/\overline{S}$ is $\mathcal F$-torsion-free.
  
  We have shown that $\overline{S}$ is a $\mathcal K$-submodule extending
  $S$, but still must show that it is the least such. For this it suffices
  to observe that, from the definition of $\overline{S}$,
  if $m\in\overline{S}$, then for any submodule $T\leq M$ satisfying
  $S\leq T\leq \overline{S}$ we have that $m/T$ is an $\mathcal F$-torsion
  element of $M/T$.
  \end{proof}

For the next theorem, which is the main result of the section,
we need another fact from \cite{kearnes92}.
In Theorem~4.1 of that paper it is shown that a quasivariety
is RCM if it satisfies the `extension principle'
and the `relative shifting lemma'. The second of these
properties will hold for any subquasivariety of an RCM quasivariety.
Thus, since $R$-Mod is RCM, any subquasivariety of $R$-Mod satisfies
the `relative shifting lemma'. The `extension principle' is
not typically inherited by subquasivarieties.

The {\it extension principle} for a quasivariety $\mathcal K$
of modules is the property that, for $M\in \mathcal K$,
the function mapping a submodule $S\leq M$ to its $\mathcal K$-extension
$\overline{S}$ is a lattice homomorphism from the lattice of submodules
of $M$ to the lattice of $\mathcal K$-submodules of $M$.
In the presence of the `relative shifting lemma', the extension
principle is equivalent to the {\it weak extension principle},
which asserts that if $S\cap T = 0$ for submodules $S, T\leq M$,
$M\in\mathcal K$, then $\overline{S}\cap \overline{T} = 0$.
(The equivalence of the weak and full extension principles
for quasivarieties satisfying the `relative shifting lemma'
is explained at the foot of page 482 of \cite{kearnes92}.)

Altogether, this means that a subquasivariety of $R$-Mod
is RCM iff it satisfies the weak extension principle.
We need this fact to prove the following theorem.

\begin{thm}
If $\mathcal F$ is a torsion notion for $R$-modules,
then the quasivariety $\mathcal K$ of $\mathcal F$-torsion-free
$R$-modules is RCM and the variety it generates is $R$-Mod.
  \end{thm}

\begin{proof}
  As discussed before the statement of the theorem,
  to prove that the quasivariety of $\mathcal F$-torsion-free
  modules is RCM it suffices to establish the weak extension principle.
  So choose an $\mathcal F$-torsion-free module $M\in\mathcal K$ and two
  submodules $S, T\leq M$ satisfying $S\cap T=0$.
  Let's prove that their $\mathcal K$-extensions
  $\overline{S}$ and $\overline{T}$   are disjoint.

  Choose $m\in \overline{S}\cap \overline{T}$.
  By Lemma~\ref{extension} there exist $A, B\in\mathcal F$
  such that $Am\subseteq S$ and $Bm\subseteq T$. By
  the down directedness of $\mathcal F$ there is a $C\subseteq A\cap B$
  that belongs to $\mathcal F$, and for this $C$ we have
  $Cm\subseteq Am\cap Bm\subseteq S\cap T=0$, so $m$
  is an $\mathcal F$-torsion element. This forces $m=0$, as desired.
  We conclude that the quasivariety of $\mathcal F$-torsion-free
  $R$-modules is RCM.
  
  To show that the variety generated by the $\mathcal F$-torsion-free
  $R$-modules is all of $R$-Mod, it suffices to note that the 1-generated
  free $R$-module is $\mathcal F$-torsion-free. This is the content
  of item (5) of Definition~\ref{torsion}. Thus $R\in\mathcal K$,
  so the variety generated by $\mathcal K$ is $R$-Mod.
  \end{proof}

\section{final statement}

Given a fixed ring $R$, we now know that a typical RCM
quasivariety $\mathcal K$ of $R$-modules can be described by a pair
$(I, \mathcal F)$ where $I$ is an ideal -- the annihilator of $\mathcal K$
-- and $\mathcal F$ is a torsion notion for $R/I$. This information
can be expressed entirely in terms of the left ideal structure of $R$
by replacing $\mathcal F$ with the set $\mathcal G$ defined
to consist of all $\nu^{-1}(A)$ for $A\in \mathcal F$ and
$\nu\colon R\to R/I$ the natural map. This yields the following statement.

\begin{thm}
  Let $R$ be a ring. A quasivariety $\mathcal K$ of $R$-modules is
  RCM iff there is a pair $(I,\mathcal G)$ such that
  $\mathcal K$ is the collection of $R$-modules
  satisfying $Ix=0$ and
  $Ax=0\to x=0$  for all $A\in \mathcal G$. Here we require that 
  $I$ be a two-sided ideal of $R$ and $\mathcal G$
  be a family of left ideals of $R$, each containing $I$
  and finitely generated over $I$, such that items
  (1)--(4) of Definition~\ref{torsion} hold, along with
  \begin{enumerate}
  \item[$(5)^{\prime}$] (Regularity modulo $I$ of elements of $\mathcal G$)
    If $A\in \mathcal G$, $r\in R$ and $Ar\subseteq I$, then
    $r\in I$.    \hfill $\Box$
    \end{enumerate}
\end{thm}


\bibliographystyle{plain}

\begin{thebibliography}{10}
\bibitem{baker77}
Baker, Kirby A.,
{\it Finite equational bases for finite algebras in a congruence-distributive equational class},
Advances in Mathematics,
{\bf 24} (1977),
no.\ 3, 207-243.


\bibitem{dziobiak09}
Dziobiak, Wies{\l}aw and Mar\'oti, Mikl\'os and
McKenzie, Ralph and Nurakunov, Anvar,
{\it The weak extension property
  and finite axiomatizability for quasivarieties},
Fundamenta Mathematicae,
{\bf 202}
(2009), no.\ 3,
199-223.

\bibitem{kearnes92}
Kearnes, Keith and McKenzie, Ralph,
{\it Commutator theory for relatively modular quasivarieties},
Transactions of the American Mathematical Society,
{\bf 331}
(1992), no.\ 2, 465-502.

\bibitem{kearnes98}
Kearnes, Keith and Szendrei, \'Agnes,
{\it The relationship between two commutators},
International Journal of Algebra and Computation,
{\bf 8}  (1998),
no.\ 4, 497-531.

\bibitem{kearnes15}
Kearnes, Keith and Szendrei, \'Agnes and Willard, Ross D.,
{\it A finite basis theorem for finite algebras with a difference term},
Transactions of the American Mathematical Society,
(to appear).

\bibitem{maroti04}
Mar\'oti, Mikl\'os and McKenzie, Ralph,
{\it Finite basis problems and results for quasivarieties},
Studia Logica,
{\bf 78} (2004), no.\ 1-2, 293-320.

\bibitem{mckenzie87}
McKenzie, Ralph,
{\it Finite equational bases for congruence modular varieties},
Algebra Universalis,
{\bf 24} (1987), no. 3, 224-250.

\bibitem{pigozzi88}
Pigozzi, Don,
{\it Finite basis theorems for
  relatively congruence-distributive quasivarieties},
Transactions of the American Mathematical Society,
{\bf 310} (1988), 
no.\ 2, 499-533.

\bibitem{willard00}
Willard, Ross,
{\it A finite basis theorem for residually finite,
  congruence meet-semi\-distributive varieties},
Journal of Symbolic Logic,
{\bf 65} (2000), no.\ 1, 187-200.

\end{thebibliography}

\end{document}